\documentclass{amsart}
\usepackage{amsmath}
\usepackage{amsthm}
\usepackage{listings}
\usepackage{xcolor}
\usepackage{mathtools}
\usepackage{paralist}
\usepackage{scalerel}
\usepackage{pbox}
\usepackage[noadjust]{cite}
\usepackage{thmtools,thm-restate}
\usepackage{latexsym}
\usepackage{textcomp}

\usepackage{tikz}
\usepackage{tikz-cd}
\usetikzlibrary{automata,positioning}

\newtheorem{theorem}{Theorem}[section]
\newtheorem{proposition}[theorem]{Proposition}
\newtheorem{lemma}[theorem]{Lemma}
\usepackage{amssymb}

\usepackage{enumitem}  
\usepackage{stmaryrd}

\setenumerate{label={\normalfont(\textit{\roman*})}, ref={(\textrm{\roman*})},
wide
}

\usepackage{ifthen}
\usepackage{hyperref}
\usepackage[mathscr]{euscript}
\usepackage{lineno}

\newtheorem{corollary}[theorem]{Corollary}

\newtheorem{alphatheorem}{Theorem}

\theoremstyle{definition}
\newtheorem{definition}[theorem]{Definition}
\newtheorem{example}[theorem]{Example}
\newtheorem{remark}[theorem]{Remark}

\DeclareMathOperator*{\EEE}{\scalerel*{\mathbb{E}}{\textstyle\sum}}

\newcommand{\norm}[1]{\left\lVert #1 \right\rVert}

\newcommand{\ip}[1]{\left\lfloor #1 \right\rfloor }

\newcommand{\floor}[1]{\left\lfloor #1 \right\rfloor}
\newcommand{\bra}[1]{\left( #1 \right)}

\renewcommand{\bar}{\overline}

\newcommand{\abs}[1]{\left|#1\right|}
\newcommand{\set}[2]{\left\{ #1 \ \middle| \ #2 \right\} }

\newcommand{\e}{\varepsilon}
\renewcommand{\a}{\alpha}

\newcommand{\NN}{\mathbb{N}}

\newcommand{\EE}{\mathbb{E}}
\newcommand{\ZZ}{\mathbb{Z}}
\newcommand{\RR}{\mathbb{R}}
\newcommand{\CC}{\mathbb{C}}

\DeclareMathAlphabet{\mathpzc}{OT1}{pzc}{m}{it}

\newcommand{\cA}{\mathcal{A}}

\newcounter{claimcounter}
\counterwithin*{claimcounter}{theorem}

\definecolor{fresh}{HTML}{2bb101}
\definecolor{checked}{HTML}{1e5e06}
\definecolor{double}{HTML}{5E3800}
\definecolor{external}{HTML}{a81a78}
\definecolor{later}{HTML}{0410ff}
\definecolor{minor-rev}{HTML}{d96a09}
\definecolor{major-rev}{HTML}{c90000}
\definecolor{skip}{HTML}{ffffff}%
\definecolor{normal}{HTML}{000000}%
\definecolor{final}{HTML}{ffffff}

\newcommand{\bb}{\mathbf}

\renewcommand{\subset}{\subseteq}

\newcommand*\patchAmsMathEnvironmentForLineno[1]{\expandafter\let\csname old#1\expandafter\endcsname\csname #1\endcsname
  \expandafter\let\csname oldend#1\expandafter\endcsname\csname end#1\endcsname
  \renewenvironment{#1}{\linenomath\csname old#1\endcsname}{\csname oldend#1\endcsname\endlinenomath}}\newcommand*\patchBothAmsMathEnvironmentsForLineno[1]{\patchAmsMathEnvironmentForLineno{#1}\patchAmsMathEnvironmentForLineno{#1*}}\AtBeginDocument{\patchBothAmsMathEnvironmentsForLineno{equation}\patchBothAmsMathEnvironmentsForLineno{align}\patchBothAmsMathEnvironmentsForLineno{flalign}\patchBothAmsMathEnvironmentsForLineno{alignat}\patchBothAmsMathEnvironmentsForLineno{gather}\patchBothAmsMathEnvironmentsForLineno{multline}}

\newcommand{\Z}{\mathbb{Z}}

\newcommand{\rb}[1]{\left( #1 \right)}

 \DeclareMathOperator{\Exists}{\exists}
\DeclareMathOperator{\Forall}{\forall}

\begin{document}

\author[J.\ Konieczny ]{Jakub Konieczny}
\address{Universit\'e Claude Bernard Lyon 1, CNRS UMR 5208, Institut Camille Jordan,
F-69622 Villeurbanne Cedex, France}
\email{jakub.konieczny@gmail.com}

\author[C.\ M\"ullner ]{Clemens M\"ullner}
\address{Institut f\"ur Diskrete Mathematik und Geometrie TU Wien, Wiedner Hauptstr. 8-10, 1040 Wien, Austria}
\email{clemens.muellner@tuwien.ac.at}
 \thanks{During the initial work on this paper, the first-named author worked within the framework of the LABEX MILYON (ANR-10-LABX-0070) of Universit\'e de Lyon, within the program "Investissements d'Avenir" (ANR-11-IDEX-0007) operated by the French National Research Agency (ANR). Currently, he works at the University of Oxford and is supported by UKRI Fellowship EP/X033813/1. The second-named author is supported by the Austrian-French project “Arithmetic Randomness” between FWF and ANR (grant numbers I4945-N and ANR-20-CE91-0006). For the purpose of open access, the authors have applied a Creative Commons Attribution (CC BY) licence to any Author Accepted Manuscript version arising.}

\title[Arithmetical subword complexity of automatic sequences]{Arithmetical subword complexity \\ of automatic sequences}
\date{\today}
 
\begin{abstract}
	We fully classify automatic sequences $a$ over a finite alphabet $\Omega$ with the property that each word over $\Omega$ appears is $a$ along an arithmetic progression. Using the terminology introduced by Avgustinovich, Fon-Der-Flaass and Frid, these are the automatic sequences with the maximal possible arithmetical subword complexity. More generally, we obtain an asymptotic formula for arithmetical (and even polynomial) subword complexity of a given automatic sequence $a$.
\end{abstract}


\maketitle 

\section{Introduction}
\label{sec:intro}

Automatic sequences --- that is, sequences computable by finite automata, see \cite{AlloucheShallit-book} for background --- have long been studied from the point of view of combinatorics on words. A notable property of these sequences is that their subword complexity is linear. To be more precise, for a sequence $a = (a(n))_{n=0}^\infty$ taking values in some finite alphabet $\Omega$ we define the subword complexity $p_{a}(\ell)$ to be the number of length-$\ell$ subwords that appear in $a$:
\begin{equation}\label{eq:intro:def-p_a}
	p_a(\ell) = \# \set{ x = (x(i))_{i=0}^{\ell-1} \in \Omega^\ell}
	{(\Exists n \geq 0)\ (\Forall 0 \leq i < \ell) \ a(n+i) = x(i)}.
\end{equation}
If $a$ is automatic then we have $p_a(\ell) = O_a(\ell)$, or in other words there exists a constant $C_a$ such that $p_a(\ell) \leq C_a \ell$ for all $\ell \geq 1$. In fact, in many cases the subword complexity can be computed explicitly. As an example, we may consider the Thue--Morse sequence $t$, given by 
\begin{equation}\label{eq:intro:def-TM}
	t(n) = s_2(n) \bmod 2, 
\end{equation}
where $s_2(n)$ denotes the sum of binary digits of $n$. The subword complexity of $t$ is
\begin{equation}\label{eq:intro:p_TM}
	p_t(n) = 
	\begin{cases}
  3\cdot 2^{k}+4(r-1)&{\text{if }} n = 2^k +r \text{ with } 1 \leq r \leq 2^{k-1},\\
  4\cdot 2^{k}+2(r-1)&{\text{if }} n = 2^k +r \text{ with }  2^{k-1} < r \leq 2^{k}.
	\end{cases}
\end{equation}
(This sequence appears as A005942 on OEIS \cite[A008277]{oeis}.)

However, automatic sequences can often look much more complicated along subsequences. For instance, the restriction of the Thue--Morse sequence to the squares $(t(n^2))_{n=0}^\infty$ is normal \cite{DrmotaMauduitRivat-2019}, meaning that for each $\ell \geq 1$, each subword $x \in \{0,1\}^\ell$ occurs with frequency $1/2^\ell$,
\[
	\lim_{N \to \infty} \frac{1}{N} 
	\# \set{ n < N}
	{(\Forall 0 \leq i < \ell) \ t(n+i) = x(i)} = 1/2^\ell.
\]
This result was later generalized to block-additive functions modulo $m$ by the second named author~\cite{Muellner2018}.
Moreover, a similar result applies to the restrictions of the Thue--Morse sequence to Piatetski-Shapiro sequences $(t(\ip{n^c})_{n=0}^\infty$ for $1 < c < 3/2$ \cite{MullnerSpiegelhofer-2017}. These results are conjectured to hold for larger exponents as well. It follows from \cite{Konieczny-2019-AIF} that, for a fixed value of $\ell$ and sufficiently large $N$, the restriction of the Thue--Morse sequence to a randomly chosen length-$\ell$ arithmetic progression contained in $[N] = \{0,1,\dots,N-1\}$ behaves like a random sequence in the sense that for each $\ell \geq 1$ there exists  $c(\ell) > 0$ such that for each $x = (x(i))_{i=0}^{\ell-1} \in \{0,1\}^\ell$ we have
\[
	\#\set{ (n,m) \in \NN_0^2}
	{
	\begin{array}{l}
		n+im < N \text{ and}\\
		t(n+im) = x(i) \\
		\text{for all } 0 \leq i < \ell
	\end{array}	 
	}
	= \frac{N^2}{2^{\ell+1}(\ell-1)}  + O(N^{2-c(\ell)})
\]

\newcommand{\parith}{p^{\mathrm{AP}}}

In particular, the Thue--Morse sequence has the largest possible arithmetical subword complexity --- a concept introduced by Avgustinovich, Fon-Der-Flaass and Frid \cite{AvgustinovichFlaassFrid-2000} as an analogue of the usual subword complexity, which we will presently introduce. For a sequence $a = (a(n))_{n=0}^\infty$ taking values in some finite set $\Omega$ the arithmetical subword complexity $\parith_{a}(\ell)$ is defined as the number of length-$\ell$ subwords that appear in $a$ along an arithmetic progression:
\begin{equation}\label{eq:intro:def-parith_a}
	\parith_a(\ell) = \# \set{ x \in \Omega^\ell}
	{(\Exists n \geq 0,\ m \geq 1)\ (\Forall 0 \leq i < \ell) \ a(n+im	) = x(i)}.
\end{equation}
This notion was further studied in \cite{Frid-2003-TCS,Frid-2003-WORDS,Frid-2006,AvgustinovichCassaigneFrid-2006,CassaigneFrid-2007}. We also point out that other modifications of the notion of subword complexity that have been studied include $d$-complexity \cite{Ivanyi-1987}, pattern complexity \cite{RestivoSalemi-2002} and maximal pattern complexity \cite{KamaeZamboni-2002}. In fact, we will consider an even more far-reaching generalisation, which we dub \emph{polynomial subword complexity}, counting the number of subwords which appear along polynomials of some degree $d$:
\begin{equation}\label{eq:intro:def-p_<d}
	p^{\leq d}_{a}(\ell) = \# \set{ x \in \Omega^\ell}
	{
	\begin{array}{l}
	(\Exists P(X) \in \RR[X])\ P(\NN_0) \subset \NN_0,\ \deg P \leq d,\\
	(\Forall 0 \leq i < \ell) \ a(P(i)) = x(i)
	\end{array}	
	}.
\end{equation}
We point out that, because the definition of $\parith_a$ in \eqref{eq:intro:def-p_a} includes the requirement that $m \neq 0$, setting $d = 1$ in \eqref{eq:intro:def-p_<d} we do not exactly recover $\parith_a$. Nonetheless, we have $\parith_{a}(\ell) \leq p^{\leq 1}(\ell) \leq \parith_{a}(\ell) + \#\Omega$. Moreover, we have the chain of inequalities
\[
	p_{a}(\ell) \leq \parith_{a}(\ell) \leq p^{\leq 1}_{a}(\ell) \leq p^{\leq 2}_{a}(\ell) \leq p^{\leq 3}_{a}(\ell) \leq \dots \leq \#\Omega^\ell.
\] 
We will say that a sequence $a$ taking values in a finite set $\Omega$ has \emph{maximal arithmetical subword complexity} if
\[
\parith_a(\ell) = \#\Omega^\ell \text{ for all } \ell \geq 1.
\]
As mentioned above, this property is enjoyed by the Thue--Morse sequence. More examples can also be found in \cite{AvgustinovichFlaassFrid-2000} and \cite{Frid-2003-TCS}.

At the opposite extreme, it is possible for an automatic sequence to have very low arithmetical or polynomial subword complexity. As a basic example, if $a$ is periodic with period $q \geq 1$ then for each $d \geq 1$ we have
\[ 
	p^{\leq d}_{a}(\ell) \leq q^{d+1},
\]
which is a direct consequence of the fact that the sequence $\a_0 + \a_1 i + \dots \a_d i^d \bmod q$ ($i \in \NN_0$) is completely determined by its initial $d+1$ terms. A less trivial example concerns (forwards) synchronising sequences, that is, automatic sequences $a$ with the property that there exists a word $w \in \Sigma_k^*$ such that for $u,v \in \Sigma_k^*$, the value $a([uwv]_k)$ depends only on $v$ (and hence is equal to $a([wv]_k)$. Here, $\Sigma_k = \{0,1,\dots,k-1\}$ denotes the set of base-$k$ digits, $\Sigma_k^*$ denotes the set of all words over $\Sigma_k$, and $[u]_k$ for $u \in \Sigma_k^*$ denotes the integer encoded by $u$. 
For such sequences it was proved by Deshouillers, Drmota, Shubin, Spiegelhofer and the second author in~\cite{Deshouillers2022},
\begin{align}\label{eq:intro:DDMSS}
	p^{\leq d}_{a}(\ell) \leq \exp{o(\ell)},
\end{align}
which is in stark contrast to the behaviour of other automatic sequences such as the Thue--Morse sequence.
This was an important intermediate goal to being able to study the subword complexity of synchronizing automatic sequences along $\floor{n^c}$ or more generally along Hardy sequences of polynomial growth, which was the main motivation in~\cite{Deshouillers2022}.
Finally, we mention backwards synchronising sequences, that is, automatic sequences $a$ with the property that there exists a word $w \in \Sigma_k^*$ such that for $u,v \in \Sigma_k^*$, the value $a([uwv]_k)$ depends only on $u$ (and hence is equal to $a([uw]_k)$. While arithmetical subword complexity of such sequences has not been previously studied, in analogy with the results in \cite{Deshouillers2022} it will not come as a surprise that \eqref{eq:intro:DDMSS} holds also in this case.

Our goal is to obtain a description of the asymptotic behaviour of the arithmetic (and polynomial) subword complexity for an arbitrary automatic sequence $a$. Motivated by the examples mentioned above, we introduce the family $\mathcal{AP}_k$ consisting of all sets $P$ of the form
\begin{equation}\label{eq:def:AP_k}
	P = \set{ n \in \NN_0}{
	\begin{array}{l}
	 \text{ the base-$k$ expansion of $n$ begins with $u$, ends with $v$,} \\
	 \text{ has length $\equiv \ell \bmod m$, and $n \equiv c \bmod q$}
	 \end{array}
	 }
\end{equation}
where $u,v \in \Sigma_k^*$, $u$ does not begin with $0$, $0 \leq \ell < m$ and $0 \leq c < q$ are integers with $q$ coprime to $k$. We think of these sets as a generalisation of the notion of a residue class, but additionally accounting for the behaviour of the base-$k$ expansion. We note that this notion is stable under change of base in the sense that if $K$ is a power of $k$ then $\mathcal{AP}_K \subset \mathcal{AP}_k$ and each set in $\mathcal{AP}_k$ is a finite union of sets in $\mathcal{AP}_K$.

With this notion in place, we are ready to introduce the parameter which controls the arithmetical subword complexity of an automatic sequence. 

\begin{definition}\label{def:eff-alph}
	Let $a = (a(n))_{n=0}^\infty$ be a $k$-automatic sequence taking values in a finite set $\Omega$. The \emph{effective alphabet size} of $a$ is the largest integer $r$ with the property that there exists $P \in \mathcal{AP}_k$ such that $a$ takes at least $r$ different values on each $Q \in \mathcal{AP}_k$ with $Q \subset P$.
	\end{definition}

In fact, since in the definition above we can freely replace $P$ with a smaller element of $\mathcal{AP}_k$, if $a$ is a $k$-automatic sequence with effective alphabet size $r$ then there exists $\Theta \subset \Omega$ with $\# \Theta = r$ and $P \in \mathcal{AP}_k$ such that for each $Q \in \mathcal{AP}_k$ with $Q \subset P$ the set of values taken by $a$ on $Q$ is precisely $\Theta$. Additionally, $r$ is the largest integer with this property.

The rationale behind the name ``effective alphabet size'' is that, with notation as in Definition \ref{def:eff-alph}, for each $\e > 0$ we can find a partition $\NN_0 = P_1 \cup P_2 \cup \dots \cup P_N \cup Q_1 \cup Q_2 \cup \dots \cup Q_M$ into elements of $\mathcal{AP}_k$ such that $a$ takes at most $r$ distinct values of each $P_i$ ($1 \leq i \leq N$) and $\bar{d}(Q_1 \cup Q_2 \cup \dots \cup Q_M) < \e$. (We do not prove this result since we do not rely on it, but it is not hard to obtain it using techniques used in this paper.) Thus, up to a negligible error, one can think of $a$ as the result of ``glueing together'' sequences on alphabets of size $r$.
	
\begin{example}
\begin{enumerate}
\item If $a$ is periodic, forwards synchronising or backwards synchronising, then $r(a) = 1$.
\item If $a$ takes the form $a(n) = F(b(n),c(n))$ for $k$-automatic sequences $b$ and $c$ and a map $F$, then $r(a)  \leq r(b)r(c)$.
\item If $r(a) = 1$ then $a$ is \emph{strongly structured} in the sense of \cite{ByszewskiKoniecznyMullner}, see Section \ref{sec:Prelim:Gowers} for further discussion.
\end{enumerate}
\end{example}

We are now ready to state our main result.

\begin{alphatheorem}\label{thm:main-B}
	Let $a = (a(n))_{n=0}^\infty$ be a $k$-automatic sequence with effective alphabet size $r$ (cf.\ Definition \ref{def:eff-alph}). Then for each $d \in \NN$ we have
	\begin{align*}
		r^{\ell} \leq \parith_a(\ell) \leq p^{\leq d}_{a}(\ell) \leq r^{\ell}\exp\bra{o(\ell)}.
	\end{align*}
\end{alphatheorem}

\begin{remark}
	Using similar methods, one could obtain a more precise upper bound of the form $r^{\ell}\exp\bra{O(\ell^\eta)}$ for some $\eta < 1$; cf.\ Remark \ref{remark:error}. For the sake of exposition, we prove a slightly weaker bound $r^{\ell}\exp\bra{o(\ell)}$, which allows us to avoid some technical computations.
\end{remark}

As alluded to earlier, a particularly interesting case of Theorem \ref{thm:main-B} is when $r(a) = \# \Omega$, meaning that the sequence $a$ has maximal arithmetical subword complexity.

\begin{corollary}\label{cor:char-MASC}
	Let $a = (a(n))_{n=0}^\infty$ be a $k$-automatic sequence taking values in a finite set $\Omega$. Then $a$ has maximal arithmetical subword complexity if and only if there exists $P \in \mathcal{AP}_k$ such that for each $Q \in  \mathcal{AP}_k$ with $Q \subset P$ and each $x \in \Omega$ there exists $n \in Q$ with $a(n) = x$.
\end{corollary}

The criterion in Corollary \ref{cor:char-MASC} may come across as somewhat complicated. In the case where the alphabet $\Omega$ has two elements, we have a simple sufficient condition.

\begin{corollary}\label{cor:main-A}
	Let $a = (a(n))_{n=0}^\infty$ be a $k$-automatic sequence taking values in $\{0,1\}$ and suppose that there exists $\a \in (0,1)$ such that for all $A > B \geq 0$ we have
\[
	\frac{1}{N} \sum_{n=0}^{N-1} a(An+B) \to \a \text{ as } N \to \infty.
\]
	Then $a$ has maximal arithmetical subword complexity.
\end{corollary}
By the same token, if $a$ instead takes values in a finite set $\Omega$ then for $a$ to have maximal arithmetical subword complexity it is enough that for each $x \in \Omega$ there exists $\a_x \in (0,1)$ such that for all $A > B \geq 0$ we have
\[
	\frac{1}{N} \# \set{ n < N}{ a(An+B) = x} \to \a_x \text{ as } N \to \infty.
\]
If the condition above is true, one might say that the sequence $a$ is \emph{totally asymptotically equidistributed} with respect to the measure on $\Omega$ given by $(\a_x)_{x \in \Omega}$.

Inspecting the proof of Theorem \ref{thm:main-B} we also notice that the conditions therein guarantee not only maximal arithmetical subword complexity but also positive frequency of all subwords.
\begin{corollary}\label{cor:main-B}
	Let $a = (a(n))_{n=0}^\infty$ be a $k$-automatic sequence with maximal arithmetical subword complexity, taking values in a finite set $\Omega$. Then for each $\ell \geq 1$ and $x \in \Omega^\ell$ we have 
\[
	\liminf_{N \to \infty} \frac{1}{N^2} 
	\#\set{ (n,m) \in \NN_0^2 }
	{
		\begin{array}{l}
		n+im < N \text{ and }
		a(n+im) = x(i)\\
		\text{for all } 0 \leq i < \ell
		\end{array}
	} > 0.
\]
\end{corollary}

\subsection*{Acknowledgements} The authors wish to thank Anna Frid for helpful comments on arithmetical subword complexity, and to Boris Adamczewski for inspiring conversations. 

\subsection*{Notation}

We let $\NN = \{1,2,\dots\}$ denote the set of positive integers and put $\NN_0 = \NN \cup \{0\}$. For $N \in \NN$ we let $[N] = \{0,1,2,\dots,N-1\}$ denote the length-$N$ initial interval of $\NN_0$.
We usually let $k$ denote the base in which we work; thus $k$ is an integer with $k \geq 2$. We let $\Sigma_k = \{0,1,\dots,k-1\}$ denote the set of base-$k$ digits, and $\Sigma_k^*$ denote the set of words over $\Sigma_k$. For $u \in \Sigma_k$ we let $\abs{u}$ denote the length of $u$. For $u \in \Sigma_k$, $[u]_k \in \NN_0$ denotes the corresponding integer, and for $n \in \NN_0$, $(n)_k \in \Sigma_k^*$ denotes the base-$k$ expansion of $n$ without any leading zeros. In particular, $(0)_k = \epsilon$ is the empty word, and $[(n)_k]_k = n$ for all $n \in \NN_0$. 
We use standard asymptotic notation, such as $O(\cdot)$ and $\gg$.
This includes the notation $f(n) = \Theta(g(n))$ for $f(n) = O(g(n))$ and $g(n) = O(f(n))$.

\section{Preliminaries} \label{sec:Prelim}

\subsection{Automata}\label{sec:Prelim:Auto}
A \emph{deterministic $k$-automaton with output} (or simply an \emph{automaton} if there is no risk of confusion) is a sextuple $\cA = (S,s_0,\Sigma_k,\delta,\Omega,\tau)$ where $S$ is a finite set of states, $s_0 \in S$ is the initial state, $\delta \colon S \times \Sigma_k \to S$ is the transition function, $\Omega$ is the output set, and $\tau \colon S \to \Omega$ is the output function. We extend $\delta$ to a map $\delta \colon S \times \Sigma_k^* \to S$ by declaring $\delta(s,uv) = \delta(\delta(s,u),v)$. The automaton $\cA$ computes the sequence $a_{\cA}$ given by $a(n) = \tau(\delta(s_0,(n)_k))$.

The automaton $\cA$ is \emph{strongly connected} if the underlying directed graph enjoys this property, meaning that for each $s,s' \in S$ there exists $u \in \Sigma_k^*$ such that $\delta(s,u) = s'$. The automaton $\cA$ is \emph{primitive} if it is strongly connected and $\gcd\bra{ \set{\abs{u}}{u \in \Sigma_k^*,\ \delta(s_0,u) = s_0} } = 1$. A \emph{strongly connected component} of $\cA$ is a maximal subset of states such that the corresponding directed graph is strongly connected. A \emph{final component} is a strongly connected component from which no other strongly connected component is reachable.

The automaton $\cA$ is \emph{synchronising} if there exists a state $s_1 \in S$ and a word $w \in \Sigma_k^*$ (sometimes called a synchronising word) such that $\delta(s,w) = s_1$ for all $s \in S$. An automatic sequence is (forwards) synchronising if it is produced by a synchronising automaton. Likewise, a sequence is backwards synchronising if it is produced by a synchronising automaton reading input starting with the least significant digit. Alternative characterisations of these notions, already mentioned in the introduction, are given in \cite[Lem.\ 3.2]{ByszewskiKoniecznyMullner}.

A set $A \subset \NN_0$ is $k$-automatic if the corresponding indicator function $1_A$ is automatic. We let $d_{\log}(A) = \lim_{N \to \infty} {1}/{(\log N)} \sum_{n=0}^{N-1} 1_A(n)/(n+1)$ denote the logarithmic density of a set $A$; the logarithmic density exists for all automatic sets.

\subsection{Higher order Fourier analysis}\label{sec:Prelim:Gowers}

Our argument hinges on a decomposition constructed by J.\ Byszewski and the authors in \cite{ByszewskiKoniecznyMullner}. For a map $f \colon [N] \to \CC$ and an integer $d \geq 1$ we define the corresponding Gowers norm
\[
	\norm{f}_{U^{d}[N]} = 
	\bra{
		\EEE_{\bb {n}}
		\prod_{\omega \in \{0,1\}^d}\mathcal{C}^{\abs{\omega}} f(n_0 + \omega_1 n_1 + \dots + \omega_d n_d)
		}^{1/2^{d}},
\]
where $\EE_{\bb n}$ denotes the average over all $\bb n = (n_0,n_1,\dots,n_s) \in \ZZ^{d+1}$ such that $n_0 + \omega_1 n_1 + \dots + \omega_d n_d \in [N]$ for all $\omega = (\omega_1,\omega_2,\dots,\omega_d) \in \{0,1\}^d$, $\abs{\omega}$ denotes the number of indices $i$ such that $\omega_i = 1$, and $\mathcal{C}$ denotes the  complex conjugation. For a comprehensive discussion on Gowers norms we refer to \cite{Tao-book}. A brief introduction, adapted to the current application, can also be found in \cite[Sec.\ 2]{ByszewskiKoniecznyMullner}. With this piece of notation in hand, we are ready to state the main result of \cite{ByszewskiKoniecznyMullner}

\begin{theorem}\label{thm:BKM}
	Let $a = (a(n))_{n=0}^\infty$ be a complex-valued $k$-automatic sequence. Then there exists a decomposition $a = a_{\mathrm{str}} + a_{\mathrm{uni}}$ where:
\begin{enumerate}
	\item $a_{\mathrm{uni}}$ is Gowers uniform in the sense that for each $s \geq 1$ there exists $c(s) > 0 $ such that $\norm{a_{\mathrm{uni}}}_{U^{s+1}[N]} = O(N^{-c(s)})$ as $N \to \infty$;
	\item $a_{\mathrm{str}}$ is structured in the sense that there exists an integer $K$ that is a power of $k$, a periodic sequence $a_{\mathrm{per}}$ with period coprime to $k$, a $K$-automatic forwards synchronising sequence $a_{\mathrm{fs}}$, and a $K$-automatic backwards synchronising sequence $a_{\mathrm{bs}}$, taking values in some finite sets $\Omega_{\mathrm{per}}$, $\Omega_{\mathrm{fs}}$, $\Omega_{\mathrm{bs}}$ respectively, as well as a map $F \colon \Omega_{\mathrm{per}} \times \Omega_{\mathrm{fs}} \times \Omega_{\mathrm{bs}} \to \CC$, such that $a_{\mathrm{str}}(n) = F(a_{\mathrm{str}}(n),a_{\mathrm{fs}}(n),a_{\mathrm{bs}}(n))$ for all $n \geq 0$.
\end{enumerate} 
\end{theorem} 

In general, the structured part $a_{\mathrm{str}}$ of an automatic sequence can be somewhat complicated. However, in \cite{ByszewskiKoniecznyMullner} we showed that $a_{\mathrm{str}} = 0$ almost everywhere (i.e., $\#\set{n < N}{a_{\mathrm{str}}(n) \neq 0}/N \to 0$ as $N \to \infty$) if and only if for all integers  $A > B \geq 0$ we have
\[
	\frac{1}{N} \sum_{n=0}^{N-1} a(An+B) \to 0 \text{ as } N \to \infty.
\]

The following lemma elucidates the connection between the concept of effective alphabet size from Definition \ref{def:eff-alph} and the structured part of an automatic sequence in Theorem \ref{thm:BKM}. Recall that the family $\mathcal{AP}_k$ consists of sets of the form \eqref{eq:def:AP_k}.

\begin{lemma}\label{lem:FCAE-positive}
	Let $A \subset \NN_0$ be a $k$-automatic set and let $R \in \mathcal{AP}_k$. Then the following conditions are equivalent.
	\begin{enumerate}
	\item\label{it:18:A} There exists $P \in \mathcal{AP}_k$ with $P \subset R$ such that for each $Q \in \mathcal{AP}_k$ with $Q \subset P$ we have $A \cap Q \neq \emptyset$.
	\item\label{it:18:B} There exists $P \in \mathcal{AP}_k$ with $P \subset R$ such that for each $Q \in \mathcal{AP}_k$ with $Q \subset P$ we have $d_{\log}(A \cap Q) > 0$.
	\item\label{it:18:D} We have $d_{\log}(A \cap R) > 0$.
	\item\label{it:18:C} There exists $P \in \mathcal{AP}_k$ with $P \subset R$ such that $1_{A,\mathrm{str}}$ is constant and strictly positive on $P$.
	\end{enumerate}
\end{lemma}
\begin{proof}
	Replacing $A$ with $A \cap R$, we may freely assume that $R = \NN_0$. We will prove implications \ref{it:18:A} $\Rightarrow$ \ref{it:18:B},  \ref{it:18:D} $\Rightarrow$ \ref{it:18:C} $\Rightarrow$ \ref{it:18:B}. Since the implications \ref{it:18:B} $\Rightarrow$ \ref{it:18:A} and \ref{it:18:B} $\Rightarrow$ \ref{it:18:D} are immediate, this will finish the argument.

	\ref{it:18:A} $\Rightarrow$ \ref{it:18:B}: Let $P$ be as in \ref{it:18:A}. For the sake of contradiction, suppose that for some $Q \in \mathcal{AP}_k$ we have $d_{\log}(A \cap Q) = 0$. Since the set $A \cap Q$ is automatic, it follows that there exists a word $w \in \Sigma_k^*$ which does not appear in the base-$k$ expansion of any $n \in A \cap Q$ (see e.g.\ \cite[Lem.\ 3.1]{ByszewskiKoniecznyMullner}). However, there exists $Q' \in \mathcal{AP}_k$ with $Q' \subset Q$ such that $w$ appears in the base-$k$ expansion of all $n \in Q'$. It follows that $A \cap Q' = \emptyset$, contradicting \ref{it:18:A}.
	
	\ref{it:18:D} $\Rightarrow$ \ref{it:18:C}: For each $\e > 0$ there exists a partition
	\[ 
	\NN_0 = P_1 \cup P_2 \cup \dots \cup P_N \cup E 
	\]
	where for each $P_i \in \mathcal{AP}_k$ and $1_{A,\mathrm{str}}$ is constant on $P_i$ for $1 \leq i < N$ and $d_{\log}(E) < \e$. Picking $\e = d_{\log}(A)$, we see that there exists some cell $P = P_i$ in the partition above such that $d_{\log}(A \cap P) > 0$ and $1_{A,\mathrm{str}}$ is constant on $P$ and takes some value $\a \in [0,1]$. It remains to show that $\a > 0$. Because $1_P$ is strongly structured and hence asymptotically orthogonal to all Gowers uniform functions (cf.\ \cite[Prop.\ 2.5]{ByszewskiKoniecznyMullner}), we have 
	\begin{align*}
	0 < d_{\log}(A \cap P)  &= \lim_{N \to \infty} \frac{1}{\log N}\sum_{n=0}^{N-1} \frac{1_P(n) 1_{A}(n)}{n+1}
	\\ &= \lim_{N \to \infty} \frac{1}{\log N}\sum_{n=0}^{N-1} \frac{\a 1_P(n)}{n+1} + \frac{1_P(n) 1_{A,\mathrm{uni}}(n)}{n+1} = \a d_{\log}(P).
	\end{align*}
	In particular, $\a =  d_{\log}(A \cap P)/d_{\log}(P) > 0$.
	
	\ref{it:18:C} $\Rightarrow$ \ref{it:18:B}: Let $P$ be as in \ref{it:18:C}, and let $\a > 0$ be the value that $1_{A,\mathrm{str}}$ takes on $P$. For $Q \in \mathcal{AP}_k$ with $Q \subset P$ we have, by the same computation as above, $d_{\log}(A \cap Q) = \a d_{\log}(Q) > 0$.
\end{proof}

We now have all the ingredients necessary to see that Corollary \ref{cor:main-A} follows from Theorem \ref{thm:main-B}.

\begin{proof}[Proof of Corollary \ref{cor:main-A}]
	It follows from the criterion for vanishing of the structured part, mentioned earlier in this section, that $a_{\mathrm{str}}$ is almost everywhere constant and takes a value strictly between $0$ and $1$. Hence, bearing in mind that $1_{a^{-1}(1)} = a$ and $1_{a^{-1}(0)} = 1-a$, we conclude from Lemma \ref{lem:FCAE-positive} that $r(a) = 2$, as needed.
\end{proof}

\section{Lower bound}\label{sec:suff}

In this section, we prove the lower bound in Theorem \ref{thm:main-B}, that is, $\parith_a(\ell) \geq r^\ell$. This is a standard application of the tools of higher order Fourier analysis. A key ingredient in the argument is the following variant of the generalised von Neumann theorem, see e.g.\ \cite[Ex.\ 1.3.23]{Tao-book} or \cite[Prop.\ 2.1]{ByszewskiKoniecznyMullner}. Below, we call a map $f \colon X \to \CC$ \emph{1-bounded} if $\norm{f}_\infty := \sup_{x \in X} \abs{f(x)} \leq 1$.

\begin{proposition}\label{prop:suff:vNeumann}
	Fix $\ell \geq 1$. Let $N \geq 1$ and let $f_0,f_1,\dots,f_{\ell-1} \colon [N] \to \CC$ be $1$-bounded maps. Then
	\[
		\abs{ \sum_{n,m = 0}^{N-1} \prod_{i=0}^{\ell-1} f_i(n+im) }
		\ll N^2 \min_{0 \leq i < \ell} \norm{f_i}_{U^{\ell-1}[N]}.
	\]
\end{proposition}
(Above, the constant implicit in the $\ll$ notation is allowed to depend on $\ell$.) As an immediate corollary, we have the following counting lemma.

\begin{lemma}\label{lem:suff:counting}
	Fix $\ell \geq 1$. Let $N \geq 1$, $\e > 0$, let $f_0,f_1,\dots,f_{\ell-1} \colon [N] \to \CC$ be $1$-bounded and assume that for each $0 \leq i < \ell$ we have a decomposition $f_i = f_{i,\mathrm{str}} + f_{i,\mathrm{uni}}$ where $f_{i,\mathrm{str}} \colon [N] \to \CC$ are $1$-bounded and $\norm{f_{i,\mathrm{uni}}}_{U^{\ell-1}[N]} \leq \e$. Then
	\[
		 \sum_{n,m = 0}^{N-1} \prod_{i=0}^{\ell-1} f_i(n+im)
		= \sum_{n,m = 0}^{N-1} \prod_{i=0}^{\ell-1} f_{i,\mathrm{str}}(n+im) + O(\e N^2).
	\]
\end{lemma}

We are now ready to approach the proof of the lower bound. Let $a$ be a $k$-automatic sequence with effective alphabet size $r$, and fix $\ell \in \NN$. Pick $P \in \mathcal{AP}_k$ such that $a$ takes at least $r$ values on each $Q \in \mathcal{AP}_k$, $Q \subset P$. Note that each $P' \in \mathcal{AP}_k$ with $P' \subset P$ also enjoys the property mentioned above. Replacing $P$ with a some $P' \in \mathcal{AP}_k$ with $P' \subset P$, we can  assume that $a$ takes on $P$ exactly $r$ different values $\omega_1,\omega_2,\dots,\omega_r$. By Lemma \ref{lem:FCAE-positive} we can construct a sequence $P_1,P_2,\dots,P_r \in \mathcal{AP}_k$ with $P \supset P_1 \supset P_2 \supset \dots \supset P_r$ such that for each $1 \leq i \leq r$, the sequence $1_{a(\omega_i)^{-1},\mathrm{str}}$ is constant on $P_i$ and takes some strictly positive value $\a_i$. 

Put $Q := P_r$, $\Theta := \{ \omega_1,\omega_2,\dots,\omega_r\}$, $\delta := \displaystyle \min_{1 \leq i \leq r} \a_i^\ell$ and let $N$ be a large integer. For $x = (x(i))_{i=0}^{\ell-1} \in \Theta^\ell$, consider the set
\begin{align*}
	S(x,N) &:= 
	\set{ (n,m) \in \NN_0^2 }
	{
		\begin{array}{l}
		n+im < N,\ n+im \in Q, \text{ and}\\
		a(n+im) = x(i) 
		\text{ for all } 0 \leq i < \ell
		\end{array}
	}.
\end{align*}
We will show that $\# S(x,N) \gg N^2$, which for sufficiently large $N$ implies that $\# S(x,N) > N$ and consequently $x$ appears in $a$ along an arithmetic progression. It will follow that $\parith_{a}(\ell) \geq \# \Theta^\ell = r^\ell$, as needed. This estimate also yields Corollary \ref{cor:main-B}.

If follows from Theorem \ref{thm:BKM} combined with e.g.\ \cite[Prop.\ 2.5]{ByszewskiKoniecznyMullner} that there is a positive constant $c$ such that $\norm{1_Q 1_{a^{-1}(\omega),\mathrm{uni}} }_{U^{\ell-1}[N]} \ll N^{-c}$  for all $\omega \in \Omega$.  
It follows from Lemma \ref{lem:suff:counting} that 
\begin{align*}
	\# S(x,N)
	 &=  \sum_{n,m = 0}^{N-1} \prod_{i=0}^{\ell-1} \bra{  1_{Q \cap [N] } 1_{a^{-1}(x(i))} } (n+im) 
	\\ &
	= \sum_{n,m = 0}^{N-1} \prod_{i=0}^{\ell-1} 1_{Q \cap [N]}(n+im) 1_{a^{-1}(x(i)),\mathrm{str}}(n+im) + O(N^{2-c} )
	\\&
	\geq 
	\delta \cdot \# \set{ (n,m) \in \NN_0^2 }
	{
		n+im \in Q \cap [N] \text{ for all } 0 \leq i < \ell
	} 
	 + O(N^{2-c} )
	\\&\gg N^2.
\end{align*}
Thus, the argument is complete.

\section{Upper bound}

In this section, we prove the upper bound in Theorem \ref{thm:main-B}. It will be convenient to first consider the special case where the sequence is primitive. 
A key idea behind our argument is to construct an alterative description of the effective alphabet size $r(a)$, which is stated as Proposition \ref{prop:alph:char-of-r}.

\subsection{Primitive case}

Let $a$ be a primitive automatic sequence produced by an automaton $\cA = (S,s_0,\Sigma_k,\delta,\Omega,\tau)$ which ignores the leading zeros (i.e., $\delta(s_0,0) = s_0$). 

We will need the notion of the height of a substitution (taken from~\cite{Queffelec2010}):
Let us consider a primitive substitution $\eta: \Lambda \to \Lambda ^k$ with a fixed point $u \in \Lambda^\infty$ (i.e. $\eta(u) = u$).
The height measures in some sense how far $u$ is from being a periodic sequence.
For every $n \geq 0$ we put
\begin{align*}
	R_n &= \{d \geq 1: u(n+d) = u(n)\} &\text{and} &&g_n &= \gcd R_n.
\end{align*}

\begin{definition}\label{def:alpha:height}
The height of $\eta$, denoted by $h = h(\eta)$, is the number
	\begin{align*}
		h = \max\set{m\geq 1 }{ \gcd(m,k) = 1,\ m \mid g_0}.
	\end{align*}
\end{definition}

We list some standard properties of the height, which can be found in~\cite{Queffelec2010}.
\begin{proposition}
	\begin{enumerate}
		\item For each  $n \geq 0$ we have
		\[h = \max\set{m\geq 1}{ \gcd(m,k) = 1,\ m \mid g_n}.\]
		\item If $h = \#\Lambda$ then $u$ is periodic.
		\item For each $0 \leq j < h$ we consider the class
		\begin{align*}
			C_j = \{u(n): n\equiv j \bmod h\}.
		\end{align*}
		These classes form a partition of $\Lambda$.
		If we identify in $u$ the letters in the same class $C_j$, we thus obtain a periodic sequence, and $h$ is the largest positive integer coprime to $k$ with this property.
	\end{enumerate}
\end{proposition}

Let $t(n) = \delta(s_0,(n)_k)$ denote the state reached by $\cA$ upon reading $n$ as input. Note that the sequence $t = (t(n))_{n=0}^\infty$ is produced by the same automaton as $a$, with the output function replaced by the identity map. The sequence $t$ is also the fixed point of the substitution $S \to S$ given by $s \mapsto (\delta(s,0),\delta(s,1),\dots,\delta(s,k-1))$. With the sets $C_j$ defined as above, we let $j(n)$ denote the unique index such that $t(n) \in C_{j(n)}$. We point out that the sequence $j = (j(n))_{n=0}^\infty$ is periodic (in fact, $j(n) = n \bmod h$). Replacing $k$ (and hence also $\eta$) with a suitable power, we may freely assume that $k \equiv 1 \bmod h$.

We will need the following technical lemma. A similar argument can be found in \cite{Mullner-2017}.

\begin{lemma}\label{lem:alph:path-exists}
	Let $q,m > 0$ be integers, with $q$ coprime to $k$. Then for each $n \in h \ZZ$ and $\ell \in \NN$ there exists $u \in \Sigma_k^*$ such that $\delta(s_0,u) = s_0$, $\abs{u} \equiv \ell \bmod m$ and $[u]_k \equiv n \bmod q$.
\end{lemma}
\begin{proof}
	We may assume that $h \mid q$. For $s,s' \in S$ and $\ell \in \ZZ/m\ZZ$ consider the set
	\[
		W(s,s';\ell) = \set{ [u]_k \bmod q}{ \delta(s,u) = s', \abs{u} \equiv \ell \bmod m} \subset \ZZ/q\ZZ.
	\]
	Since $\cA$ is primitive, all sets $W(s,s';\ell)$ are non-empty. The composition rule for $\delta$ implies that 
	\begin{equation}\label{eq:alph:70:00}
	W(s,s';\ell)k^{\ell'} + W(s',s'';\ell') \subset W(s,s'';\ell+\ell')
	\end{equation}		
	for all $s,s',s'' \in S$ and $\ell,\ell'  \in \ZZ/m\Z$. Comparing the cardinalities of the sets in \eqref{eq:alph:70:00}, we see that the inclusion is in fact an equality:
	\begin{equation}\label{eq:alph:70:01}
	W(s,s';\ell)k^{\ell'} + W(s',s'';\ell') = W(s,s'';\ell+\ell').
	\end{equation}		
	Setting $s=s'=s''$ and $\ell = \ell' = 0$ in \eqref{eq:alph:70:01}, we see that $W(s,s;0)$ is a subgroup $H = m\ZZ/q\ZZ \subset \ZZ/q\ZZ$, where $m \mid q$ is independent of $s$. In general, $W(s,s';\ell) = H + w(s,s';\ell)$ is a coset of $H$ (here, $w(s,s';\ell) \in \ZZ/q\ZZ$).  Thus, \eqref{eq:alph:70:01} becomes
	\begin{equation}\label{eq:alph:70:02}
	w(s,s';\ell)k^{\ell'} + w(s',s'';\ell') \equiv w(s,s'';\ell+\ell') \bmod m.
	\end{equation}	
	Recall that $\delta(s_0,0) = s_0$ and consequently $0 \in W(s_0,s_0,1)$ and $w(s_0,s_0,1) \equiv 0 \bmod m$. By the same token, $w(s_0,s_0,\ell) \equiv 0 \bmod m$ for all $\ell$. As a consequence,
	\[
		\set{ [u]_k \bmod q}{ \delta(s_0,u) = s_0} = \bigcup_{\ell \in \ZZ/m\ZZ} W(s_0,s_0;\ell) = H.
	\]
	It follows that $H = h\ZZ/q\ZZ$, and the argument is complete.
%
\end{proof}

The second ingredient which we will need comes from \cite{Mullner-2017}. Let $c$ denote the least possible cardinality of the set $\delta(S,w) = \set{\delta(s,w)}{ s \in S}$ for $w \in \Sigma_k^*$, and let $\mathcal{M} = \{M_0,M_1,\dots,M_{p-1}\}$ denote the family of all possible sets of the form $\delta(S,w)$ with cardinality $c$.
For $n \geq 0$, let $i(n)$ denote the unique index such that $\delta(M_0,(n)_k) = M_{i(n)}$. Without loss of generality, we may assume that $s_0 \in M_0$, which implies that $t(n) \in M_{i(n)}$ for all $n$.

Finally, for $0 \leq i < p$ and $0 \leq j < h$ we let $S_{i,j} = M_i \cap C_j$. We point out that for all $n$ we have $t(n) \in S_{i(n),j(n)}$.

\begin{example}
	Let us take $k = 3$ and consider the automaton $\cA$ depicted by the following diagram:
	


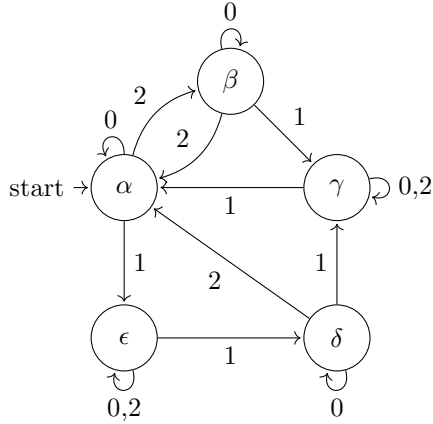
\begin{center}
\begin{tikzpicture}[scale=.5,shorten >=1pt,node distance=2cm, on grid, auto] 
   \node[initial, state] (a)   {$\alpha$}; 
   \node[state] (b) [above right=of a] {$\beta$};
   \node[state] (d) [below right=of b] {$\gamma$};
   \node[state] (c) [below =of d] {$\delta$}; 
   \node[state] (e) [below =of a] {$\epsilon$}; 
   
  \tikzstyle{loop}=[min distance=6mm,looseness=4]
    \path[->]     
    (a) edge [out = 90, in = 120, min distance= 6 mm,looseness=4, above] node  {0} (a);
    \path[->]  
    (b) edge [loop above] node  {0} (b);
    \path[->]  
    (c) edge [loop below] node  {0} (c);
    \path[->]  
    (d) edge [loop right] node  {0,2} (d);
    \path[->]  
    (e) edge [loop below] node  {0,2} (e);
    
    \tikzstyle{path}=[min distance=10mm,in=270,out=270, looseness=1]
    \path[->]  
    (a) edge [out = -90, min distance = 10mm, looseness = 1.75, in = 90] node {1} (e);
      \tikzstyle{loop}=[min distance=6mm,looseness=4]
    \path[->]  
    (a) edge [bend left] node {2} (b);
    
    \path[->]  
    (b) edge node {1} (d);
    \path[->]  
    (b) edge [bend left] node [above left] {2} (a);

     \path[->]
     (c) edge node {1} (d);
     \path[->]
     (c) edge node {2} (a);
     
     \path[->]
     (d) edge node {1} (a);
     
     \path[->]
     (e) edge node [below] {1} (c);
    
\end{tikzpicture}
\end{center}	

We compute the corresponding automatic sequence
\begin{align*}
	\alpha, \epsilon, \beta, \epsilon, \delta, \epsilon, \beta, \gamma, \alpha, \epsilon, \delta, \epsilon, \delta, \gamma, \alpha,\ldots
\end{align*}
which shows $R_0 = \{8, 14, \ldots\}$ and $g_0 \mid 2$, which implies $h \mid 2$. Moreover, considering the sets $C_0 = \{\alpha, \beta, \delta\}$ and $C_1 = \{\gamma, \epsilon\}$ one finds $h = 2$. 

Moreover, we have $c = 4$ and $M_0 = \{\alpha,\beta,\gamma,\epsilon\}$ and $M_1 = \{\alpha,\gamma,\delta,\epsilon\}$. 
The sets $S_{i,j}  = M_i \cap C_j$ are given by:
\begin{align*}
	S_{0,0} &= \{\alpha,\beta\}, \qquad S_{0,1} = \{\gamma,\epsilon\},\\
	S_{1,0} &= \{\alpha,\delta\}, \qquad S_{1,1} = \{\gamma,\epsilon\}.
\end{align*}

\end{example}

\begin{proposition}\label{prop:alph:char-of-r}
	With the same notation as above, we have
	\[
		r(a) = \max_{i,j} \#\{\tau(s): s \in S_{i,j}\}.
	\]
\end{proposition}
\begin{proof}
	The inequality 
	\begin{equation}\label{eq:alph:30:05}
		r(a) \leq \max_{i,j} \#\{\tau(s): s \in S_{i,j}\}
	\end{equation}
	is relatively simple. Indeed, since $(i(n))_n$ is synchronising and $(j(n))_n$ is periodic, for each $P \in \mathcal{AP}_k$ we can find $Q \in \mathcal{AP}_k$ with $Q \subset P$ and $0 \leq i < p$ and $0 \leq j < h$ such that $i(n) = i$ and $j(n) = j$ for all $n \in Q$. 
	Hence, $t(n) \in S_{i,j}$ and $a(n) \in \{\tau(s): s \in S_{i,j}\}$ for all $n \in Q$, which implies that $r(a) \leq \#\{\tau(s): s \in S_{i,j}\}$, and \eqref{eq:alph:30:05} follows.
	
	It remains to prove the reverse inequality
	\begin{equation}\label{eq:alph:30:06}
		r(a) \geq \max_{i,j} \#\{\tau(s): s \in S_{i,j}\}.
	\end{equation}
	Pick a minimal set $M_i \in \mathcal{M}$ and a residue $j \bmod h$ ($0 \leq i < p$, $0 \leq j < h$). There exists $P \in \mathcal{AP}_k$ such that for all $n \in P$ we have $i(n) = i$ and $j(n) = j$. We will show that for each $Q \in \mathcal{AP}_k$ with $Q \subset P$ and each $s \in S_{i,j}$ there exists $n \in Q$ such that $t(n) = s$. Since $i$ and $j$ were arbitrary, once this is accomplished, \eqref{eq:alph:30:06} will follow. 
	
	Replacing $Q$ with a smaller set if needed, we can assume that $Q$ takes the form 
	\[ 
		Q = \set{n \in \NN_0}{ n \equiv j' \bmod q,\ (n)_k \in u\Sigma_K^*v}
	\]
	for some $K$ that is a power of $k$, $u,v \in \Sigma_k^*$ and some $0 \leq j' < q$ with $q$ coprime to $k$. (Here and elsewhere, we identify $\Sigma_K$ with $\Sigma_k^{\log_k K}$.) Without loss of generality we may assume that $h \mid q$ and thus $j' \equiv j \bmod h$. Prolonging $u$ if necessary, we may freely assume that $\delta(s_0,u) = s_0$. It follows from the minimality of $M_i$ and the fact that all states in $S$ are reachable from $s_0$ that there exists $w \in \Sigma_k^*$ such that $\delta(s_0,wv) = s$; indeed, otherwise $\delta(S,v)$ would be a proper subset of $M_i$. By primitivity, we can assume that $w \in \Sigma_K^*$. Now, the definition of $C_j$ implies that $[u]_k + [w]_k + [v]_k \equiv j \bmod h$.
	 By Lemma \ref{lem:alph:path-exists}, for each $m \in \ZZ$ we can find $w_m' \in \Sigma_K^*$ of some length $\ell(m)$ divisible by $\log_k K$ and such that $\delta(s_0,w_m') = s_0$, $[w_m']_k \equiv h m \bmod q$ and $k^{\ell(m)} \equiv 1 \bmod q$. It remains to note that for all $m$ we have $\delta(s_0,u w_m' w v) = s$ and there exists $m$ such that $[u w_m' w v]_k \equiv j' \bmod q$.  
\end{proof}

\begin{proposition}
	Let $a$ be a primitive automatic sequence. Then
	\begin{align*}
		p_a^{\leq d} (\ell) \leq r^{\ell} p_{(i,j)}^{\leq d} (\ell) \leq r^{\ell} \exp\rb{o(\ell)}.
	\end{align*}
\end{proposition}
\begin{proof}
	We recall that from the construction of $S_{i,j}$ that for all $n$ we have $t(n) \in S_{i(n),j(n)}$. In particular $a(n) \in \{\tau(s): s \in S_{i(n),j(n)}\}$.
	This already shows the first inequality. 
	For the second inequality we let $h$ denote the height of $a$ and have
	\begin{align*}
		 p_{(i,j)}^{\leq d} (\ell) \leq p_{i}^{\leq d} (\ell) \cdot  p_{j}^{\leq d} (\ell) \leq \exp(o(\ell)) \cdot h^{d+1},
	\end{align*}
	where the second inequality follows from Proposition 5.8 in~\cite{Deshouillers2022}. 
\end{proof}
\begin{remark}\label{remark:error}
The upper bound in Proposition 5.8 in~\cite{Deshouillers2022} can be improved by balancing the error terms more carefully.
We switch for this remark to the notation used in~\cite{Deshouillers2022} (i.e. $\ell$ is replaced by $H$).
If we let $\lambda$ grow with $H$ and ignore the estimates using $\varepsilon$, all the arguments can be kept essentially unchanged and we find the upper bound
\begin{align*}
	p^{\leq d}_{a}(H) &\ll \abs{\operatorname{Ker}_k(a)} \cdot \abs{\mathcal{A}}^{k^{\lambda(d+1)}} \cdot \binom{k^{\lambda}}{O(k^{\lambda(1-\eta)})} \cdot \abs{\mathcal{A}}^{(H/k^{\lambda}+1) k^{\lambda(1-\eta)}}\\
		&\ll \abs{\mathcal{A}}^{k^{\lambda(d+1)}} \cdot (k^{\lambda})^{O(k^{\lambda (1-\eta)})} \cdot \abs{\mathcal{A}}^{(H/k^{\lambda}+1) k^{\lambda(1-\eta)}}
\end{align*}
Choosing for example $k^{\lambda} = \floor{H^{1/(d+2)}}$ leads to
\begin{align*}
	p^{\leq d}_{a}(H) \leq \exp\rb{O\rb{H^{1-\eta/(d+2)}}}.
\end{align*}
\end{remark}

\subsection{General case}

We now deal with sequences that are not necessarily primitive. To begin with, we will need the following lemma.

\begin{lemma}\label{lem:alpha:partition}
	Let $w \in \Sigma_k^*$ be a word, let $P$ be a degree $d$ polynomial such that $P(\NN_0) \subset \NN_0$, and let $\ell \in \NN$. Then there exists $0 < \theta < 1$, dependent only on $\abs{w}$ and $d$, such that $[0, \ell)$ can be covered by $\ell^\theta$ intervals, each of which either contains exactly one integer point, or is contained in a set of the form $P^{-1}([m k^i, (m+1)k^i))$ where $m,i \in \NN_0$ and $w$ is a subword of $(m)_k$.
\end{lemma}
\begin{proof}
Let $P(n) = \alpha_0 + \alpha_1 n + \ldots + \alpha_d n^d$ and $P_1(n) = P(n) - \alpha_0$.
It follows from Lemma 5.6 in~\cite{Konieczny2023} that
\begin{align*}
	M := \max_{n \in [0,\ell]} \abs{P_1(n)}  = \Theta\rb{ \max_{1\leq j \leq d} \abs{\ell^j \alpha_j}}.
\end{align*}
Similarly we have 
\begin{align*}
		M' := \max_{n \in [0,\ell]} \abs{P'(n)}  = \Theta\rb{ \max_{1\leq j \leq d} \abs{j \ell^{j-1} \alpha_j}} = \Theta(M/\ell).
\end{align*}
Moreover, for each $\delta > 0$ we have
\begin{align}\label{eq:alph:87:05}
	\lambda\bra{\set{x \in [0,\ell] }{ \abs{P'(x)} < \delta^{d-1}M'}} \ll \delta \ell.
\end{align}
We note that~\eqref{eq:alph:87:05} also holds for $d = 1$ as the left hand side equals $0$.

Let $\e > 0$ be a small positive quantity, to be determined in the course of the argument, and let
\begin{align*}
I = \set{x \in [0, \ell] }{ \abs{P'(x)} > \varepsilon^{d-1} M'}.
\end{align*}
We note that $I$ is a union of at most $d$ intervals and it follows from \eqref{eq:alph:87:05} that
\begin{align}\label{eq:alph:87:06}
\lambda( [0,\ell] \setminus I) \ll \varepsilon \ell.
\end{align}
(In the case where $d = 1$ we have $P'(n) = M' = \a_1$, whence \eqref{eq:alph:87:06} is trivially true and $I = [0,\ell]$.) 

Let $R > 0$ be a large real number, to be determined in the course of the argument, and let $K = k^i$ be a power of $K$ such that $K \leq M/R < k K$. Recall that $P([0,\ell])$ is an interval of length at most $2M$, and hence can be covered with $O(R)$ intervals of the form $J_m := [mK,(m+1)K)$. For each $m$, the set $I \cap P^{-1}(J_m)$ is a union of $O(d)$ intervals. If the base-$k$ expansion $(m)_k$ contains $w$ as a subword then these intervals satisfy the required conditions; we will call such intervals ``good''.

We cover the remaining part of $[0,\ell)$ with singletons. Thus, it remains to estimate the number of integers in $[0,\ell)$ not covered by ``good'' intervals. These integers fall into two categories. Firstly, we have the elements of $[0,\ell) \setminus I$, whose number can be estimated by \eqref{eq:alph:87:05}. Secondly, we have the ``bad'' intervals, corresponding to the intervals $J_m$ such that $(m)_k$ does not contain $w$. The number of such ``bad'' values of $m$ is $O(R^{1-\lambda})$ for some $\lambda > 0$ (dependent only on $k$ and $\abs{w}$). Recall that for each $m$ the set $I \cap P^{-1}(J_m)$ is a union of $O(d)$ intervals, each of length 
$O({\ell}/\bra{\e^{d-1} R}+1)$. 
Thus, in total, the number of points and ``bad'' intervals we obtain is, up to a constant, bounded by
\begin{align*}
	\varepsilon \ell + R^{1-\lambda} \cdot \frac{ \ell}{\e^{d-1} R} + R^{1-\lambda}
	=  \ell \cdot \bra{ \varepsilon + 1/\bra{\e^{d-1} R^\lambda}} + R^{1-\lambda}.
\end{align*}
Optimising, we are lead to choose
\begin{align*}
	R &= \ell^{\frac{d}{d+\lambda}}, &
	\e &= R^{- \frac{\lambda}{d}} = \ell^{- \frac{\lambda}{d+\lambda}}.
\end{align*}
This finishes the argument, with $\theta = \frac{d}{d+\lambda}$.
\end{proof}

We are now ready to prove Theorem \ref{thm:main-B} in full generality.
\begin{proof}[Proof of Theorem \ref{thm:main-B}, non-primitive case]
	Let $\cA = (S,s_0,\Sigma_k,\delta,\Omega,\tau)$ be an automaton which computes $a$. Replacing $k$ with a power, we may freely assume that $\delta(s,00) = \delta(s,0)$ for all $s \in S$.
	It follows directly from Definition \ref{def:eff-alph} that
\begin{align*}
	r = r(a) = \max \set{r(a')}{a' \text{ is computed by a final component of } \cA}.
\end{align*}
	Recall that we have already proved that for each sequence $a'$ computed by a final component of $\cA$, starting from a state $s_0'$ with $\delta(s_0',0) = s_0'$, we have 
\begin{align*}
	p_{a'}(\ell) 
	\leq r(a')^{\ell}\exp(f(\ell))  
	\leq r^{\ell}\exp(f(\ell)),
\end{align*}
where $f \colon \NN \to \RR_{\geq 0}$ is some function with $f(\ell)/\ell \to 0$ as $\ell \to \infty$. We may freely assume that $f$ is defined on $[0,\infty)$, $f(0) = 0$ and that $f$ is concave.\footnote{If $f$ is not concave, consider the area $A = \set{(x,y)}{ x \geq 0,\ 0 \leq y \leq f(x)}$ below the graph of $f$. The closure of the convex hull of $A$ is the area below the graph of some concave $g$. Directly from construction, $g$ is concave and $g \geq f$. It is routine to verify that $g(\ell)/\ell \to 0$ as $\ell \to \infty$.}
Let $w$ be a word such that $\delta(s,w)$ belongs to a final component of $\cA$ for each $s \in S$ (cf. \cite[Lem.\ 3.1]{ByszewskiKoniecznyMullner}). Replacing $w$ with $w0$ if necessary, we may further assume that for each $s \in S$ the state $s' := \delta(s,w)$ belongs to a final component and satisfies $\delta(s',0) = s'$.

Let $\ell$ be a large integer and let $P$ be a degree $d$ polynomial with $P(\NN_0) \subset \NN_0$. By Lemma \ref{lem:alpha:partition}, we can partition $[0,\ell)$ into $R \ll \ell^{\theta}$ intervals $I$ that either are singletons or are contained in $P^{-1}([m k^i,(m+1)k^i))$ for some $m,i \in \NN_0$ such that $w$ is a subword of $(m)_k$. In the latter case, for $n \in I$ we have
\[
	a(P(n)) =  a'( P'(n))
\]	
for a sequence $a'$ computed by a strongly connected component of $\cA$, starting from a state $s_0'$ with $\delta(s_0',0) = s_0'$ and a polynomial $P'$ with $P'(\NN_0) \subset \NN_0$. (To be more precise, we can find integers $m',i' \in \NN_0$ such that $(m')_k$ ends with $w$ and $[m k^i, (m+1) k^i ) \subset [m' k^{i'}, (m'+1) k^{i'} )$; thus replacing $m,i$ with $m',i'$ we may freely assume that $(m)_k$ ends with $w$. Let $s_0' = \delta(s_0,(m)_k)$, let $a'$ be the sequence computed by $\cA$ starting from the state $s_0'$, and let $P'(n) = P(n) - m k^i$. Then $a(P(n)) = a'(P'(n))$ for $n \in I$. One remaining problem is that $P'$ could take negative values outside of $I$. To overcome it, we replace $P'(n)$ with $P'(n) + h k^j$, where $j$ is sufficiently large that $k^j > (m+1) k^i$ and $h > 0$ is an integer such that $\delta(s_0',(h)_k) = s_0'$, which exists by strong connectivity.) 
Above, $\theta \in (0,1)$ is a constant which depends only on $\abs{w}$ and $d$. The number of partitions, as described above, is $\ell^{O(R)}$. Fix one such partition. For each singleton $\{n\}$ we have at most $\# \Omega$ possible values of $a(P(n))$. For each non-degenerate interval of length $\ell_i$ the number of possible values taken by $a(P(n))$ is at most $r^{\ell_i}\exp(f(\ell_i))$. Note that, by concavity, we have $\sum_{i} f(\ell_i) \leq R f(\ell/R)
 $, where the sum runs over all lengths $\ell_i$ of non-degenerate intervals involved in the partition. In total, we obtain the estimate
\begin{align*}
	p_{a}(\ell) &\leq \ell^{O(R)} \cdot \#\Omega^R \cdot r^{\ell} \cdot \exp\bra{R f(\ell/R)}
	 \\&= r^{\ell} \cdot \exp\bra{R(O(\log(\ell) + f(\ell/R))}
	 = r^{\ell} \cdot \exp\bra{o(\ell)}
	 .
\end{align*}
(In the last transition, we used the fact that $\ell/R \to \infty$ and consequently also $Rf(\ell/R)/\ell \to 0$ as $\ell \to \infty$.)
\end{proof}


\bibliographystyle{alphaabbr}
\bibliography{bibliography}
\end{document}